\documentclass[leqno,11pt]{amsart} 
\usepackage{amssymb, amsmath, color, amstext, tikz}
\usepackage[english]{babel}
\usepackage{appendix}
\newtheorem{letterthm}{Theorem}

\newtheorem{theo}{Theorem}[section]
\newtheorem{lemma}[theo]{Lemma}
\newtheorem{prop}[theo]{Proposition}
\theoremstyle{definition} 
\newtheorem{defi}[theo]{Definition}
\newcommand{\diagtaul}{\tau_l(x)=\begin{tikzpicture}[baseline=-.1cm]\draw (-.25,-.25)--(.25,-.25)--(.25,.25)--(-.25,.25)--(-.25,-.25);\draw (0,.25) arc (0 : 180 : .25);\draw (-.5,-.25) arc (-180 : 0 : .25);\node at (0,0) { $x$ };\draw (-.5,-.25)--(-.5,.25);\end{tikzpicture}}
\newcommand{\diagtaur}{\tau_r(x)=\begin{tikzpicture}[baseline=-.1cm]\draw (-.25,-.25)--(.25,-.25)--(.25,.25)--(-.25,.25)--(-.25,-.25);\draw (.5,.25) arc (0 : 180 : .25);\draw (0,-.25) arc (-180 : 0 : .25);\node at (0,0) { $x$ };\draw (.5,-.25)--(.5,.25);\end{tikzpicture}}
\newcommand{\diagxdag}{\begin{tikzpicture}[baseline = .425cm]\draw (.25,.25)--(.75,.25)--(.75,.75)--(.25,.75)--(.25,.25);\draw (.5,0)--(.5,.25);\draw (.5,.75)--(.5,1);\draw (0,.5)--(.25,.5);\draw (.75,.5)--(1,.5);\node at (.5,.5) { $x^*$ };\node at (.85,.85) {\tiny\$};\end{tikzpicture}}
\newcommand{\diagmult}{xy = \sum_{a=0}^{ \min(2n,2m) }\begin{tikzpicture}[baseline=.4cm]
\draw (-.75,0)--(-.25,0)--(-.25,.5)--(-.75,.5)--(-.75,0);\draw (.75,0)--(.25,0)--(.25,.5)--(.75,.5)--(.75,0);\draw (-.625,.5)--(-.625,1);\draw (.625,.5)--(.625,1);\draw (.375,.5) arc (0 : 180 : .375);\node at (-.5,.25) { $ x $ };\node at (.5,.25) { $ y $ };\node at (0 , 1) {{ \scriptsize{$a$} }};\end{tikzpicture}}
\newcommand{\diagmultbullet}{x\bullet y = \begin{tikzpicture}[baseline=.4cm]\draw (-.75,0)--(-.25,0)--(-.25,.5)--(-.75,.5)--(-.75,0);\draw (.5,0)--(0,0)--(0,.5)--(.5,.5)--(.5,0);\draw (-.5,.5)--(-.5,1);\draw (.25,.5)--(.25,1);\node at (-.5,.25) { $ x $ };\node at (.25,.25) { $ y $ };\end{tikzpicture}}
\newcommand{\AAP}{absorbing amenability property}
\newcommand{\AOP}{asymptotic orthogonality property}
\newcommand{\cusu}{cup subalgebra}
\newcommand{\FGF}{free group factor}
\newcommand{\maam}{maximal amenable}
\newcommand{\MASA}{maximal abelian subalgebra}
\newcommand{\plal}{planar algebra}
\newcommand{\suplal}{subfactor planar algebra}
\newcommand{\SAOP}{strong asymptotic orthogonality property}
\newcommand{\sufa}{subfactor}
\newcommand{\vna}{von Neumann algebra}

\newcommand{\AM}{A\subset M}
\newcommand{\bu}{\bullet}
\newcommand{\C}{\mathbb C}
\newcommand{\ep}{\varepsilon}

\newcommand{\gO}{\geqslant 0}
\newcommand{\LA}{L^2(A)}
\newcommand{\LM}{L^2(M)}
\newcommand{\limni}{\lim_{n\rightarrow\infty}}
\newcommand{\loriar}{\longrightarrow}
\newcommand{\Pl}{\mathcal P}
\newcommand{\spann}{\text{Span}}

\begin{document}
\title{The cup subalgebra has the absorbing amenability property}
\maketitle
\begin{center}
{\sc by Arnaud Brothier\footnote{Vanderbilt University, Department of Mathematics, 1326 Stevenson Center Nashville, TN, 37240, USA,\\ arnaud.brothier@vanderbilt.edu} and Chenxu Wen \footnote{Vanderbilt University, Department of Mathematics, 1326 Stevenson Center Nashville, TN, 37240, USA,\\ chenxu.wen@vanderbilt.edu}}
\end{center}

\begin{abstract}\noindent
Consider an inclusion of diffuse \vna s $\AM$.
We say that $\AM$ has the \AAP\ if for any diffuse subalgebra $B\subset A$ and any amenable intermediate algebra $B\subset D\subset M$ we have that $D$ is contained in $A.$
We prove that the \cusu\ associated to any \suplal\ has the \AAP.
\end{abstract}

\section*{Introduction and main results}
Amenability is a fundamental concept in various area of mathematics.
Connes proved the striking result that a \vna\ is amenable if and only if it is hyperfinite \cite{Connes_76_classification_inj_factors}.
In this article, we study \maam\ subalgebras.
Fuglede and Kadison showed that any II$_1$ factor contains a \maam\ \sufa\ \cite{Fuglede_Kadison_ConjMvN}.
Popa exhibited the first example of an abelian \maam\ subalgebra of a II$_1$ factor, thus giving a counter-example to a question of Kadison \cite{popa_max_inj}.
He defined the notion of \AOP\ (AOP) and showed that a singular \MASA\ (masa) with the AOP is \maam.
Many other examples have been given using the same strategy \cite{Ge_maxinj,Shen_maxinj,Cameron_Fang_Ravichandran_White_10_max_inj,Brot_GJSW_max_inj,Houdayer_12_maxinj,Boutonnet_Carderi_maxamI}.
A completely new strategy in proving maximal amenability has been given in \cite{Boutonnet_Carderi_maxamII}.

Peterson conjectured that any maximal amenable subalgebra of a \FGF\ is the unique maximal amenable extension of any of its diffuse subalgebra.
Inspired by this question and the work of Houdayer on maximal Gamma extensions \cite{Houdayer_Bogoljubov,Houdayer_Gamma}, we consider the notion of \AAP\ (AAP).
An inclusion of \vna s $\AM$ has the AAP if for any diffuse subalgebra $B\subset A$ and any amenable intermediate algebra $B\subset D\subset M$ we have that $D$ is contained in $A.$
In particular, if $A$ is amenable, then it is maximal amenable.
Houdayer proved that the generator masa has the AAP \cite{Houdayer_Gamma}.
The second author showed that the radial masa has the AAP \cite{Wen_Radial}.

In this article, we present a new class of examples that have the AAP.
Those examples are constructed with Jones \plal s \cite{Jones_planar_algebra}.
If $\Pl$ is a \suplal, then we can associate to it a II$_1$ factor $M$ \cite{GJS_1}.
This II$_1$ factor is isomorphic to an interpolated \FGF\ $L(\mathbb F_t)$ where $t$ is a linear combination of the index and the global index of $\Pl$ \cite{Dykema_LF_t,Radulescu_Random_matrices_LF_t,GJS_2,Hartglass_13_GJS}.
This factor admits a generic abelian subalgebra $\AM$ that we call the \cusu. 
The first author previously proved that the \cusu\ is \maam\ \cite{Brot_GJSW_max_inj}.
We prove the following theorem:
\begin{letterthm}\label{theo:A}
The \cusu\ associated to any \suplal\ has the \AAP.
\end{letterthm}
This provides many examples of subalgebras of interpolated \FGF s with the AAP.
Note, it is still unknown if there exists a \suplal\ such that its associated \cusu\ is isomorphic to the generator or the radial masa.

\subsection*{Acknowledgement}
We express our gratitude to Cyril Houdayer and Jesse Peterson for encouragements and making comments on an earlier version of this manuscript.

\section{Preliminaries}

\subsection{Planar algebras}

A \plal\ is a collection of complex $*$-algebras $\Pl=(\Pl_n^\pm:n\gO)$ on which the set of shaded planar tangles acts.
See \cite{Jones_planar_algebra,jones_planar_algebra_II} for more details.
We follow similar conventions that was used in \cite{CJS} for drawing a shaded planar tangle.
We decorate strings with natural numbers to indicate that they represent a given number of parallel strings.
The distinguished interval of a box is decorated by a dollar sign if it is not at the top left corner.
We do not draw the outside box and will omit unnecessary decorations.
The left and right traces of a \plal\ are the maps $\tau_l:\Pl_n^\pm\loriar\Pl_0^\mp$ and $\tau_r:\Pl_n^\pm\loriar\Pl_0^\pm$ defined for any $n\gO$ such that 
$$\diagtaul \text{ and } \diagtaur \text{ for any } x\in \Pl_n^\pm.$$
Suppose that $\Pl_0^\pm=\C.$
The \plal\ is called spherical if the two traces agree on each element of $\Pl.$
We say that $\Pl$ is non-degenerate if the sesquilinear forms $(x,y)\mapsto \tau_l(xy^*)$ and $(x,y)\mapsto \tau_r(xy^*)$ are positive definite.
A \suplal\ is a \plal\ such that each space $\Pl_n^\pm$ is finite dimensional, $\Pl_0^\pm=\C$, $\Pl$ is spherical and non-degenerate.
The modulus of a \suplal\ is the value of a closed loop.

\subsection{Construction of a II$_1$ factor}

We recall a construction due to Jones et al. \cite{JSW}.
Consider the direct sum
$Gr\Pl=\bigoplus_{n\gO}\Pl^+_n$
that we equipped with the following Bacher product and involution:
$$\diagmult, \text{ and } x^\dag=\diagxdag, \text{where } x\in \Pl_n^+ \text{ and } y\in \Pl^+_m.$$
Consider the linear form $\tau: Gr\Pl\loriar\C$ that sends $x\in\Pl_0^+$ to itself and $0$ to any element in $\Pl_n^+$ if $n\neq 0.$
The vector space $Gr\Pl$ endowed with those operation is an associative $*$-algebra with a faithful tracial state.
Let $H$ be the completion of $Gr\Pl$ for the inner product $(x,y)\mapsto \tau(xy^*)$.
The left multiplication of $Gr\Pl$ on $H$ is bounded and defines a $*$-representation \cite{GJS_1,JSW}.
Let $M$ be the \vna\ generated by $Gr\Pl$ inside $B(H).$
It is an interpolated \FGF\ \cite{GJS_2,Hartglass_13_GJS}.
We define another multiplication on $Gr\Pl$ by requiring that if $x\in\Pl_n$ and $y\in\Pl_m$, then 
$$
\diagmultbullet \in \Pl^+_{n+m}.
$$
Denote by $x^{\bullet n}$ the n-th power of $x$ for this multiplication.
Remark, $\Vert a\bullet b\Vert_2=\Vert a\Vert_2\Vert b\Vert_2$, for all $a\in \Pl_n$ and $b\in\Pl_m$.
Therefore, this multiplication is a continuous bilinear form for the $L^2$-norm $\Vert\cdot\Vert_2$ of $M$.
We extend this operation on $L^2(M)\times L^2(M)$ and still denote it by $\bullet$.

\subsection{The cup subalgebra}

Let $\cup$ be the unity of the $*$-algebra $\Pl_1^+$, viewed as an element of $M$ \cite{GJS_1}.
Let $\AM$ be the von Neumann subalgebra generated by $\cup.$
We call it the \cusu.

\subsection{Strong asymptotic orthogonal property}

Popa introduced the notion of \AOP\ (AOP) in \cite{popa_max_inj}.
We consider a strengthening of this notion which was used by Houdayer and the second author \cite{Houdayer_Bogoljubov,Wen_Radial}.

\begin{defi}
Let $\AM$ be a diffuse subalgebra of a tracial \vna.
This inclusion has the \SAOP\ (SAOP) if for any free ultrafilter $\omega$ and any diffuse subalgebra $B\subset A$ we have 
$$xy\perp yx \text{ for any } x\in (M^\omega\cap B')\ominus A^\omega \text{ and } y\in M\ominus A.$$
\end{defi}

Note, a diffuse subalgebra $\AM$ has the SAOP if and only if it has the AOP relative to all of its diffuse subalgebras in the sense of Houdayer \cite[Definition 5.1]{Houdayer_Bogoljubov}.

The following theorem is an extension of a theorem of Popa \cite{popa_max_inj}.
\begin{theo}\label{theo:SAOP}\cite[Theorem 8.1]{Houdayer_Bogoljubov}
If $\AM$ is a diffuse subalgebra with the SAOP such that $L^2(M)\ominus\LA$ is a mixing $A$-bimodule (e.g. is a direct sum of the coarse bimodule $\LA\otimes \LA$), then it has the AAP.
\end{theo}
See also \cite[Proposition 2.1]{Wen_Radial}.

\section{Proof of the main theorem}

\begin{prop}\label{prop}
Let $(A,\tau)$ be a tracial \vna\ and $B\subset A$ a diffuse subalgebra.
Denote by $L^2(A)$ the Gelfand-Naimark-Segal completion of $A$ for the trace $\tau$.
Consider a sequence $\xi=(\xi_n:n\gO)$ of unit vectors of the coarse bimodule $\LA\otimes\LA.$
Suppose that for any $b\in B$ we have $\limni\Vert b\cdot \xi_n-\xi_n\cdot b\Vert=0.$
Then, if $p\in B(\LA)$ is a finite rank projection, then $\limni\Vert (p\otimes 1)\xi_n\Vert=\limni\Vert (1\otimes p)\xi_n\Vert=0.$
\end{prop}

\begin{proof}
Let $A,B,\xi,$ and $p$ as above.
It is sufficient to prove the proposition when $p$ is a rank one projection.
Let $\eta\in\LA$ be a unit vector such that $p=p_\eta$ is the rank one projection onto $\C \eta$.
Consider $0<\ep<1$ and a natural number $I$ such that $16/(I+1)<\ep.$
Since $B$ is diffuse, there exists a sequence of unitaries $(u_n)_n$ in $B$ such that $\limni \langle u_n\cdot\zeta_1,\zeta_2\rangle=0$ for any $\zeta_1,\zeta_2\in\LA.$
Consider the quantity $\delta=\max(\vert\langle u_n\cdot \eta,u_m\cdot \eta\rangle\vert: n\neq m, n,m\leqslant I)$.
By \cite[Proposition 2.3]{Houdayer_12_maxinj}, we have that 
$$\sum_{i=0}^I \Vert (p_{u_i\cdot \eta} \otimes 1) \xi_n\Vert^2\leqslant g(\delta) \Vert\xi_n\Vert^2 \text{ for any } n\gO,$$
where $g$ is a positive function satisfying $\lim_{\delta\rightarrow 0}g(\delta)=1.$
Hence, there exists a subsequence $(v_n)_n$ such that 
$$\sum_{i=0}^I \Vert (p_{v_i\cdot \eta} \otimes 1) \xi_n\Vert^2\leqslant 2\Vert\xi_n\Vert^2=2 \text{ for any } n\gO.$$
Let $\lambda:B\loriar B(\LA\otimes\LA)$ be the left action of $B$ on the coarse bimodule $\LA\otimes\LA.$
Observe, $p_{v_i\cdot \eta}\otimes 1=\lambda(v_i)\circ (p_\eta \otimes 1) \circ \lambda(v_i)^*$ and $v_i$ is a unitary, for any $i\gO.$
Therefore, $\Vert (p_{v_i\cdot \eta} \otimes 1) \xi_n\Vert= \Vert(p_\eta\otimes 1) v_i^*\cdot \xi_n \Vert$ for any $n,i\gO.$
By assumption, there exists $N>0$ such that for any $n\geqslant N$ and $i\leqslant I$ we have 
$\Vert v_i^*\cdot \xi_n - \xi_n\cdot v_i^* \Vert<\ep/4.$
Therefore,
\begin{align*}
\Vert (p_\eta\otimes 1) \xi_n \Vert & = \Vert (p_\eta\otimes 1) (\xi_n\cdot v_i^*) \Vert\\
& \leqslant \Vert (p_\eta\otimes 1) ( v_i^*\cdot \xi_n - \xi_n\cdot v_i^*) \Vert + \Vert (p_\eta\otimes 1) (v_i^*\cdot\xi_n) \Vert\\
&\leqslant \ep/4 + \Vert (p_{v_i\cdot\eta}\otimes 1) \xi_n\Vert \text{ for any } n\geqslant N, i\leqslant I.
\end{align*}
We obtain
\begin{align*}
\sum_{i=0}^I  \Vert (p_{\eta} \otimes 1) \xi_n\Vert^2 & \leqslant \sum_{i=0}^I (\ep^2/16 + \ep/2 \Vert (p_{v_i\cdot \eta}\otimes 1) \xi_n \Vert + \Vert (p_{v_i\cdot \eta}\otimes 1) \xi_n \Vert^2)\\
& \leqslant (I+1)(\ep/16 + \ep/2) + 2 \text{ for any } n\geqslant N.
\end{align*}
Therefore, $\Vert (p_{\eta} \otimes 1) \xi_n\Vert^2\leqslant \ep/16+\ep/2 + 2\ep/16\leqslant \ep$ for any $n\geqslant N.$
The same proof shows that there exists $M>0$ such that for any $n\geqslant M$ we have $\Vert (1\otimes p_\eta)\xi_n\Vert^2\leqslant \ep.$
This proves the proposition.
\end{proof}

Fix a \suplal\ $\Pl$ with modulus $\delta>1$ and denote by $\AM$ its associated \cusu.
Consider the subspace $V_n\subset \Pl_n^+,n\gO$ of elements that vanishes when they are capped off on the top left corner and vanished when they are capped off on the top right corner.
Let $V\subset \LM$ be their orthogonal direct sum.
By \cite[Theorem 4.9]{JSW}, the following map is an isomorphism of $A$-bimodules:
$$\phi:\LA\oplus(\LA\otimes V\otimes\LA)\loriar \LM, a+b\otimes v\otimes c\longmapsto a+b\bullet v\bullet c.$$
This implies that the $A$-bimodule $\LM\ominus\LA$ is isomorphic to an infinite direct sum of the coarse bimodule.
We identify $\LM$ with $\phi^{-1}(\LM).$

Consider the finite dimensional subspace $L_m=\spann(\cup^{\bullet k}:k\leqslant m)\subset A$ for $m\gO,$
where $\cup^{\bullet 0}=1\in \Pl_0^+$.
Denote by $L_m^\perp$ the orthogonal complement of $L_m$ inside $\LA$ for any $m\gO.$

\begin{lemma}\label{lemma}
Let $m\gO$ and $x\in M\cap L_m^\perp\otimes V\otimes L_m^\perp$, $y\in M\cap L_m\otimes V\otimes L_m$.
Then $xy\in L_m^\perp\otimes V\otimes L_m$ and $yx\in L_m\otimes V\otimes L_m^\perp.$
In particular, $xy\perp yx.$
\end{lemma}

\begin{proof}
Consider $x= \cup^{\bu k} \bu v\bu \cup^{\bu l}$ and $y = \cup^{\bu s} \bu w\bu \cup^{\bu t}$, where $s,t<m+1\leqslant k,l$ and $v,w\in V\cap Gr\Pl.$
We have that 
$$xy=\sum_{i=0}^{s+1} \delta^{[i/2]}\cup^{\bu k}\bu v\bu \cup^{\bu (l+s-i)} \bu w\bu \cup^{\bu t},$$
where $[i/2]=i/2$ if $i$ is even and $i/2-1/2$ if $i$ is odd.
Observe, $L_m^\perp$ is equal to the closure of $\spann(\cup^{\bu k}:k\geqslant m+1).$
Therefore, $xy\in L_m^\perp\otimes V\otimes L_m$ and similarly $yx\in L_m\otimes V\otimes L_m^\perp.$
The space $M\cap  L_m^\perp\otimes V\otimes L_m$ (resp. $M\cap  L_m\otimes V\otimes L_m^\perp$) is the weak closure of $\spann(\cup^{\bu k} \bu v\bu \cup^{\bu l}:k,l\geqslant m+1,v\in V\cap Gr\Pl)$ (resp. $\spann(\cup^{\bu s} \bu w\bu \cup^{\bu t}:s,t\leqslant m,w\in V\cap Gr\Pl)$).
This concludes the proof by a density argument.
\end{proof}

We are ready to prove the main theorem of the article.
\begin{proof}[Proof of Theorem \ref{theo:A}]
Let $\Pl$ be a \suplal, $\AM$ its associated \cusu, and $B\subset A$ a diffuse subalgebra.
Consider $x\in M^\omega\ominus A^\omega$ in the relative commutant of $B$ and $y\in M\ominus A$, where $\omega$ is a free ultrafilter on $\mathbb N.$
Let us show that $xy\perp yx.$
Observe, $Gr\Pl$ is a weakly dense $*$-subalgebra of $M.$
Therefore, we can assume that $y\in Gr\Pl$ by Kaplansky density theorem.
This implies that there exists $m\gO$ such that $y\in Gr\Pl\cap L_m\otimes V\otimes L_m.$
Let $(x_n)_n$ be a representative of $x$ in the ultrapower $M^\omega.$
We can assume that for any $n\gO$ we have $x_n\in \LM\ominus\LA.$
Let $p\in B(\LA)$ be the orthogonal projection onto $L_m$.
It is a finite rank projection.
Therefore, by Proposition \ref{prop}, $(p\otimes 1)x=(1\otimes p)x=0$. 
Hence, we can assume that $x_n\in L_m^\perp\otimes V\otimes L_m^\perp$ for any $n\gO.$
Lemma \ref{lemma} implies that $x_ny\perp yx_n$ for any $n\gO.$
This implies that $xy\perp xy.$

Theorem \ref{theo:SAOP} implies that $\AM$ has the AAP.
\end{proof}

\bibliographystyle{alpha}

\begin{thebibliography}{CFRW10}

\bibitem[BCa]{Boutonnet_Carderi_maxamI}
R.~Boutonnet and A.~Carderi.
\newblock Maximal amenable subalgebras of von {N}eumann algebras associated
  with hyperbolic groups.
\newblock {\em to appear in Math. Ann.}

\bibitem[BCb]{Boutonnet_Carderi_maxamII}
R.~Boutonnet and A.~Carderi.
\newblock Maximal amenable von {N}eumann subalgebras arising form maximal
  amenable subgroups.
\newblock {\em to appear in Geom. Funct. Anal.}

\bibitem[Bro14]{Brot_GJSW_max_inj}
A.~Brothier.
\newblock The cup subalgebra of a {II}$_1$ factor given by a subfactor planar
  algebra is maximal amenable.
\newblock {\em Pacific J. Math.}, 269(1):19--29, 2014.

\bibitem[CFRW10]{Cameron_Fang_Ravichandran_White_10_max_inj}
J.~Cameron, J.~Fang, M.~Ravichandran, and S.~White.
\newblock The radial masa in a free group factor is maximal injective.
\newblock {\em J. Lond. Math. Soc.}, 82(2):783--809, 2010.

\bibitem[CJS14]{CJS}
S.~Curran, V.F.R. Jones, and D.~Shlyakhtenko.
\newblock On the symmetric enveloping algebra of planar algebra subfactors.
\newblock {\em Trans. Amer. Math. Soc}, 366(1):113--133, 2014.

\bibitem[Con76]{Connes_76_classification_inj_factors}
A.~Connes.
\newblock Classification of injective factors.
\newblock {\em Ann. of Math.}, 104(2):73--115, 1976.

\bibitem[Dyk94]{Dykema_LF_t}
K.J. Dykema.
\newblock Interpolated free group factors.
\newblock {\em Pacific J. Math.}, 163:123--135, 1994.

\bibitem[FK51]{Fuglede_Kadison_ConjMvN}
B.~Fuglede and R.V. Kadison.
\newblock On a conjecture of {M}urray and von {N}eumann.
\newblock {\em Proc. Nat. Acad. Sci. U.S.A.}, 37:420--425, 1951.

\bibitem[Ge96]{Ge_maxinj}
L.~Ge.
\newblock On maximal injective subalgebras of factors.
\newblock {\em Adv. Math.}, 118(1):34--70, 1996.

\bibitem[GJS10]{GJS_1}
A.~Guionnet, V.F.R. Jones, and D.~Shlyakhtenko.
\newblock Random matrices, free probability, planar algebras and subfactor.
\newblock {\em Quanta of maths: {N}on-commutative Geometry Conference in Honor
  of {A}lain {C}onnes, in {C}lay {M}ath. {P}roc.}, 11:201--240, 2010.

\bibitem[GJS11]{GJS_2}
A.~Guionnet, V.F.R. Jones, and D.~Shlyakhtenko.
\newblock A semi-finite algebra associated to a planar algebra.
\newblock {\em J. Funct. Anal.}, 261(5):1345--1360, 2011.

\bibitem[Har13]{Hartglass_13_GJS}
M.~Hartglass.
\newblock Free product von {N}eumann algebras associated to graphs, and
  {G}uionnet, {J}ones, {S}hlyakhtenko subfactors in infinite depth.
\newblock {\em J. Funct. Anal.}, 256(12):3305--3324, 2013.

\bibitem[Hou14a]{Houdayer_12_maxinj}
C.~Houdayer.
\newblock A class of {II}$_1$ factors with an exotic abelian maximal amenable
  subalgebra.
\newblock {\em Trans. Amer. Math. Soc.}, 366(7):3693--3707, 2014.

\bibitem[Hou14b]{Houdayer_Bogoljubov}
C.~Houdayer.
\newblock Structure of {II}$_1$ factors arising from free {B}ogoljubov actions
  of arbitrary groups.
\newblock {\em Adv. Math.}, 260:414--457, 2014.

\bibitem[Hou15]{Houdayer_Gamma}
C.~Houdayer.
\newblock Gamma stability in free product von neumann algebras.
\newblock {\em Commun. Math. Phys.}, 336(2):831--851, 2015.

\bibitem[Jon]{Jones_planar_algebra}
V.F.R. Jones.
\newblock Planar algebras {I}.
\newblock {\em Preprint. arXiv:9909.027}.

\bibitem[Jon12]{jones_planar_algebra_II}
V.F.R. Jones.
\newblock Planar algebra course at {V}anderbilt.
\newblock {\em http://math.berkeley.edu/~vfr/}, 2012.

\bibitem[JSW10]{JSW}
V.F.R. Jones, D.~Shlyakhtenko, and K.~Walker.
\newblock An orthogonal approach to the subfactor of a planar algebra.
\newblock {\em Pacific J. Math.}, 246:187--197, 2010.

\bibitem[Pop83]{popa_max_inj}
S.~Popa.
\newblock Maximal injective subalgebras in factors associated with free groups.
\newblock {\em Adv. Math.}, 50:27--48, 1983.

\bibitem[Rad94]{Radulescu_Random_matrices_LF_t}
F.~Radulescu.
\newblock Random matrices, amalgamated free products and subfactors of the von
  {N}eumann algebra of a free group, of noninteger index.
\newblock {\em Invent. Math.}, 115:347--389, 1994.

\bibitem[She06]{Shen_maxinj}
J.~Shen.
\newblock Maximal injective subalgebras of tensor products of free group
  factors.
\newblock {\em J. Funct. Anal.}, 240(2):334--348, 2006.

\bibitem[Wen]{Wen_Radial}
C.~Wen.
\newblock Maximal amenability and disjointness for the radial masa.
\newblock {\em to appear in J. Funct. Anal.}

\end{thebibliography}

\end{document}